\documentclass[11pt]{amsart}
\usepackage{amssymb, amsfonts, latexsym, amsthm, amsmath, verbatim, enumerate}
\usepackage{soul}
\usepackage{paralist}

\newtheorem{theorem}{Theorem}
\theoremstyle{plain}

\newtheorem{case}{Case}

\newtheorem{conjecture}{Conjecture}
\newtheorem{corollary}{Corollary}

\newtheorem{definition}{Definition}

\newtheorem{lemma}{Lemma}

\newtheorem{proposition}{Proposition}

\newtheorem*{question}{Question}

\newcommand{\RR}{\mathbb{R}}

\newcommand{\Q}{\mathbb{Q}}
\newcommand{\ZZ}{\mathbb{Z}}

\newcommand{\bpm}{\begin{pmatrix}}
\newcommand{\epm}{\end{pmatrix}}

\newcommand{\vv}{\mathbf{v}}

\usepackage{fancyhdr}
\usepackage[left=1in,top=1in,right=1 in,bottom=1in]{geometry}

\begin{document}

\title{ Cubic Irrationals and Periodicity via a Family of Multi-dimensional Continued Fraction Algorithms}
\author{Krishna Dasaratha \and Laure Flapan\and Thomas Garrity \and Chansoo Lee \and 
Cornelia Mihaila\and Nicholas Neumann-Chun
 \and Sarah Peluse\and Matthew Stoffregen}

\thanks{The authors thank the National Science Foundation for their support of this research via grant DMS-0850577}
\thanks{K. Dasaratha- Stanford University} \thanks{L. Flapan- University of California, Los Angeles} \thanks{T. Garrity- Williams College, tgarrity@williams.edu} \thanks{C. Lee- University of Michigan at Ann Arbor} \thanks{C. Mihaila- University of Texas at Austin} \thanks{N. Neumann-Chun- Williams College}\thanks{S. Peluse- The University of Chicago}\thanks{M. Stoffregen- University of California, Los Angeles}

\maketitle

\begin{abstract}
We  construct a  countable family of multi-dimensional continued fraction algorithms, built out of five specific multidimensional continued fractions, and  find a wide class of cubic irrational real numbers $\alpha$ so that either $(\alpha, \alpha^2)$ or $(\alpha, \alpha-\alpha^2)$ is purely periodic with respect to an element in the family.  These cubic irrationals seem to be quite natural, as we show that, for every cubic number field, there exists a pair $(u, u')$ with $u$  a unit in the cubic number field (or possibly the quadratic extension of the cubic number field by the square root of the discriminant)  such that $(u, u')$ has a periodic multidimensional continued fraction expansion under one of the maps in the family  generated by the initial five maps. 
These results are built on a careful technical analysis of certain units in cubic number fields and our family of multi-dimensional continued fractions.  
We then recast the linking of  cubic irrationals with periodicity to the linking of cubic irrationals with the  construction of a matrix with nonnegative integer entries for which at least one row is eventually periodic.

\end{abstract}

\section{Introduction}

A real number has an eventually periodic decimal expansion precisely when it is rational.  A real number has an eventually periodic continued fraction expansion precisely when it is a quadratic irrational.  A natural question (which Hermite \cite{HermiteC} asked Jacobi in 1839) is if there is a way to represent real numbers as sequences of nonnegative integers such that a number's algebraic properties are revealed by the periodicity of its sequence. Specifically, Hermite wanted an algorithm that returns an eventually periodic sequence of integers if and only if its input is a cubic irrational.

There have been many attempts to produce such an algorithm.  These attempts fall into two classes \cite{BuchmannJ85}.  The first is a class of algorithms based on number geometric interpretations of the standard continued fractions algorithm, which goes back at least to the work of  Minkowksi.  More current examples  of such work can be found in \cite{DuboisE80, Delone64, BuchmannJ89}. The second class of algorithms tries to generalize the standard continued fractions algorithm arithmetically.  So far, no periodic arithmetic algorithm has been found.  Algorithms in this class are known as ``multidimensional continued fractions."  Descriptions of the most well-known of the Multidimensional Continued Fraction algorithms can be found in \cite{SchweigerF00, SchweigerF95, BrentjesAJ81} and in a recent survey 
\cite{Berthe11}
.

We suspect that there is no such algorithm.  Instead,  the best possible solution to the Hermite problem will be in the form of a family of algorithms,  meaning that for any arbitrary $\alpha$, we can produce a $\beta$ such that the sequence of integers associated with  $(\alpha,\beta)$ will be periodic with respect to some algorithm in the family if and only if $\alpha$ is a cubic irrational. Such a family is capable of being encoded as a matrix $(a_{ij})$, with $1\leq i,j <\infty$.  Each row will be the sequence of integers associated to some algorithm in the family.   In this language, we would want one of the rows to be eventually periodic if and only if $\alpha$ is a cubic irrational.  
This suggests that to one should not construct a single multidimensional continued fraction algorithm but instead construct a family of such algorithms.  In this paper we propose such a family of algorithms (built out of Trip maps \cite{SMALL-TRIP}) and show show that for a wide class of pairs of  cubic irrationals in the same cubic number field that there will be an element from this family  for which the pair will have periodic expansion.

Heuristically, the reason, in part,  that a real number $\alpha$ is a quadratic irrational is that the vector $(1,\alpha)$ is an eigenvector of a $2\times 2$ invertible matrix with rational entries that is not a multiple of the identity.  More precisely, there needs to be a matrix $A\in GL(2, \Q)$ such that $(1,\alpha)A$ is an eigenvector of a $2\times 2$ invertible matrix with rational entries that is not a multiple of the identity.  The continued fraction algorithm can be interpreted as  a procedure for producing a sequence of matrices in $SL(2, \ZZ)$ (and hence in $ GL(2, \Q)$) so that we are guaranteed that if $\alpha$ is a quadratic irrational, then  there is a matrix $A$ in our sequence with the  $(1,\alpha)A$ our desired type of eigenvector.  Of course, continued fractions have many other key properties, especially about questions involving Diophantine approximations.

In a similar fashion, numbers $\alpha$ and $\beta$ will be in the same cubic number field if there is matrix $A\in GL(3, \Q)$ such that $(1,\alpha, \beta)A$ is an eigenvector of a $3\times 3$ invertible matrix with rational entries that is not a multiple of the identity.   All multidimensional continued fraction algorithms produce sequences of matrices in $GL(3, \Q)$ (often in $SL(3, \ZZ)$).  Periodicity of the given algorithm corresponds to finding our desired matrix $A$.  It appears that each of the existing algorithms are, in some sense, one-dimensional, and for dimension reasons, not guaranteed to find $A$.  This supports the observation that for almost all  existing multidimensional continued fraction algorithms, eventual periodicity means that $\alpha$ and $\beta$ are in the  same cubic number field, but that for none of these algorithms can the converse be shown.  This also suggests why we turn to a five dimensional family of multidimensional continued fraction algorithms.

In section \ref{earlier} we discuss some earlier work and in particular the work of Dubois and Paysant-Le Roux  \cite{Dubois-Paysant-Le Roux75}.  In section \ref{pdsection} we review  the family of $216$ Trip maps, from \cite{SMALL-TRIP}. In section \ref{mainsection} we present the main technical results of our paper.  We construct a family of multidimensional continued fraction algorithms formed by compositions of five  of the Trip maps  and show that  there is  a wide class of cubic irrational real numbers $\alpha$ so that either $(\alpha, \alpha^2)$ or $(\alpha, \alpha-\alpha^2)$ is purely periodic with respect to an element in the family.  These types of cubic irrationals are quite natural, as they are linked to finding units in cubic number fields.  In  particular, we show that, for every cubic number field, there exists a pair $(u, u')$ with $u$  a unit in the cubic number field (or possibly the quadratic extension of the cubic number field by the square root of the discriminant)  such that $(u, u')$ has a periodic multidimensional continued fraction expansion under one of the maps in this family of $5$ maps.  This is  the technical heart of the paper, linking periodicity and units.  A similar link was done earlier by Dubois and Paysant-Le Roux but the roots of cubic polynomials that we find seem  simpler to express than the earlier work of Dubois and Payant-Le Roux  \cite{Dubois-Paysant-Le Roux75} (no doubt due to the fact that Dubois and Paysant-Le Roux are only working with the Jacobi-Perron algorithm).   We are also expressing far more cubic irrationals as periodic sequences than those in Dubois and Paysant-Le Roux. In section \ref{hermitesection}, we put the idea of finding a family of multidimensional continued fractions into the rhetoric of Hermite matrices, allowing us to explicitly express the view that a multidimensional continued fraction family will solve the Hermite problem should mean that a pair $(\alpha, \beta)$ are cubic irrationals in the same cubic number field if and only if at least one row of the corresponding Hermite matrix is eventually periodic.


We would like to thank Oleg Karpenkov for pointing out a significant misstatement of our main result in an earlier version of this paper.  Finally we would like to thank the referee for pointing out to us the earlier work of Dubois and Payant-Le Roux  \cite{Dubois-Paysant-Le Roux75}.

\section{Earlier Work}\label{earlier}

As mentioned above, for almost every existing multidimensional continued fraction algorithm we know that periodicity implies cubic irrationality, at worse.  For each of these algorithms then, we can characterize some of the cubic polynomials whose roots are described periodically.  For example, in the first paper on triangle sequences  \cite{GarrityT01}, it was shown that a pair $(\alpha, \alpha^2)$ will have purely periodic extension of length one with triangle sequence $(n,n,n, \ldots)$ (where $0<\alpha<1$) if and only if $\alpha$ is a root of 
$$x^3+nx^2+x-1=0.$$

Some of the most interesting results along these lines is in the work of Dubois and Paysant-Le Roux in {\it Algorithme de Jacobi-Perron dans les extensions cubiques} \cite{Dubois-Paysant-Le Roux75}  from almost forty years ago. Their work involves the Jacobi-Perron algorithm.   In particular, they showed that for every real cubic number field there is a pair of numbers with periodic Jacobi-Perron expansion.  Specifically, in their lemma 2, they show that if $\alpha$ is a positive root of the cubic 
$$x^3 - ax^2 + bx -1,$$
where $a$ and $b$ are positive integers satisfying $b>3$ and $a>2(b-1)$, then the pair
$$\left( \frac{(a-1)(\alpha-1)}{(a-3)\alpha - (b-3)},  \frac{(\alpha - 1)^2}{ (a-3)\alpha - (b-3)} \right)  $$
will be purely periodic with respect to the Jacobi-Perron algorithm.  In theorem 1, they show that every real cubic number field has an element $\alpha$ that is a root of a polynomial $x^3 - ax^2 + bx -1,$
where $a$ and $b$ are positive integers satisfying $b>3$ and $a>2(b-1)$

As shown in \cite{SMALL-TRIP}, this means that these pairs are periodic with respect to a  Combo Trip  map involving $T_{(e,e,e)} $  and $T_{e, (123), e)}$.  (We will define Combo Trip maps and this notation in the next section).  The main goal of this paper is to drastically expand the numbers for which we are guaranteed to have periodic Combo Trip  map expansions.  The pairs will look simpler than the pair of Dubois and Paysant-Le Roux, though this is in part due to the fact that Dubois and Paysant-Le Roux restrict themselves to the Jacobi-Perron algorithm .  We will also show that our found pairs will capture all cubic number fields.


\section{The Trip Algorithms}\label{pdsection}

This section closely follows sections 2 and 3 from \cite{SMALL-TRIP}.

We begin by describing the original triangle map, which is a multidimensional continued fraction algorithm, as defined in \cite{GarrityT01} and further developed in \cite{GarrityT05,  Schweiger05} .  (The triangle map was shown to be erogodic in \cite{SchweigerF08}.)  We then introduce permutations into the definition of the triangle map, thereby generating a family of $216$ multidimensional continued fractions.  Finally, we show how to produce triangle sequences, which are the analog of continued fraction expansions.

\subsection{The Triangle Map}
\label{trianglemap}

The ordinary continued fraction expansion is computed by iterating the Gauss map on the unit interval.  The triangle map generalizes this method.  Instead of the unit interval, we use a 2-simplex, i.e. a triangle.  We think of this triangle as lying in $\RR^3$.  Specifically, define
\[\triangle = \{(1,x,y) : 1 \geq x \geq y > 0\}\]

Define $\pi:\RR^3 -\{(0,x,y\} \rightarrow \RR^2$ by setting 
$$\pi(z,x,y)= \left(\frac{x}{z},\frac{y}{z}\right).$$
The vectors
\[
\vv_1 = \left(\begin{array}{c}1 \\0 \\0\end{array}\right), 
\vv_2 = \left(\begin{array}{c}1 \\1 \\0\end{array}\right),
\vv_3 = \left(\begin{array}{c}1 \\1 \\1\end{array}\right).\]
are the vertices of $\triangle$.  Now in order to partition $\triangle$, we consider the following two matrices:
\[A_0 = \left(\begin{array}{ccc}0 & 0 & 1 \\1 & 0 & 0 \\0 & 1 & 1\end{array}\right),
A_1 = \left(\begin{array}{ccc}1 & 0 & 1 \\0 & 1 & 0 \\0 & 0 & 1\end{array}\right)\]
Let $B = (\vv_1 \  \vv_2  \ \vv_3)$.  Then the column vectors of $BA_0$ and $BA_1$ describe a disjoint partition of $\triangle$.  We iterate this division process, and define $\triangle_k$ to be the image of $\triangle$ under $A_1^k A_0$.  Then let $T^{(k)}$ be a bijective map from $\pi(\triangle_k)$ to $\pi(\triangle)$ given by
\begin{eqnarray*}
T^{(k)}(x, y)& =& \pi\left[ (1, x, y) \cdot (B A_0^{-1} A_1^{-k} B^{-1})^T \right]\\
&=&\left( \frac{y}{x},\frac{1-x-ky}{x} \right ).
\end{eqnarray*}
For a given $(x,y)\in\triangle$, $(x,y)\in\triangle_{k}$ if and only if $k=\lfloor \frac{1-x}{y} \rfloor$. Define $T:\triangle\to\triangle$ by $T(x,y)=T^{(k)}(x,y)$ for $(x,y)\in\triangle_{k}$. This map $T$, which is called the  \textbf{triangle map}, is analogous to the Gauss map.

\subsection{Incorporating Permutations}

The triangle map consists of a process of partitioning of a triangle with vertices $(\vv_1, \vv_2,\vv_3)$ into the triangles with vertices $(\vv_2, \vv_3,\vv_1 + \vv_3)$ and \mbox{$(\vv_1,\vv_2,\vv_1+\vv_3)$}. The essential thing to note is that this process assigns a particular ordering of vertices to both the vertices of the original triangle and the vertices of the two triangles produced.  But this ordering is by no means canonical.  We can permute the vertices at several stages of the triangle division.  This leads to the following definition.
\begin{definition}For every $(\sigma, \tau_0, \tau_1) \in S^3_3$, define
$$F_0 =F_0(\sigma, \tau_0, \tau_1)= \sigma A_0 \tau_0 \text{ and } F_1 = F_1(\sigma, \tau_0, \tau_1)=\sigma A_1 \tau_1$$
by thinking of $\sigma,$ $\tau_0$, and $\tau_1$ as column permutation matrices. Further, subdivide $\triangle$ by setting  $\triangle_k(\sigma, \tau_0, \tau_1) $ to be the image of $\triangle$ under $F_1^k F_0$
\end{definition}
Thus given any $(\sigma, \tau_0, \tau_1) \in S^3_3$, we can partition $\triangle$ in a distinct way using the matrices $F_0$ and $F_1$ instead of $A_0$ and $A_1$.  This leads to the definition of a family of multidimensional continued fractions algorithms, each specified by a $(\sigma, \tau_0, \tau_1) \in S^3_3$.  Because $\left|S_3^3\right| = 216$, this family has 216 elements.

Most important for our present purposes, we can define the map analogous to the Gauss map for any of the 216 multidimensional continued fractions algorithms.
\begin{definition}\label{Tmap}
Given any $(\sigma, \tau_0, \tau_1) \in S^3_3$, define
\[T^{(k)}_{\sigma, \tau_0, \tau_1}(1,x, y) = \pi\left[ (1, x, y) \cdot (B F_0^{-1} F_1^{-k} B^{-1})^T \right],\]
for any $(x,y)\in \triangle_k(\sigma, \tau_0, \tau_1),$
and $T_{\sigma, \tau_0, \tau_1}$ in an analogous manner to the way $T$ was defined in \ref{trianglemap}. Recall that $B = (\vv_1 \  \vv_2  \ \vv_3)$ and that $\pi(a,b,c) =(1,\frac{b}{a}, \frac{c}{a} )$.
\end{definition}
Note that the matrix $(B F_0^{-1} F_1^{-k} B^{-1})^T$ is in $SL(3,\ZZ)$.
The triangle partition maps $T_{\sigma, \tau_0, \tau_1}$  are called \textbf{Trip Maps}.

\subsection{Trip Sequences}

We now use this to define a method of constructing an integer sequence.  The sequence is analogous to the continued fraction expansion of a number.

\begin{definition}
Given $(x, y)\in \pi(\triangle)$, we recursively define $a_n$ to be the non-negative integer such that $(T_{\sigma, \tau_0, \tau_1})^n(x, y)$ is in $\pi(\triangle_{a_n}(\sigma, \tau_0, \tau_1))$. The \textbf{triangle sequence} of $(x, y)$ with respect to $(\sigma, \tau_0, \tau_1)$ is $(a_0,a_1,...)$.
\end{definition}

\subsection{An Even Larger Family: Combo Trip Maps}

Finally, and most importantly, we can obtain a much larger family of algorithms by mixing and matching these maps. For example, we can carry out the first subdivision of $\triangle$ using $(\sigma, \tau_0, \tau_1)$, the second subdivision using $(\sigma', \tau_0', \tau_1')$, and so forth. As we shown in \cite{SMALL-TRIP}, many well known multidimensional continued fractions are compositions of a finite number of the $216$ permutation-division maps.  To each of these Combo Trip  maps and to each $(x,y)$ is associated a sequence of non-negative integers in analog to the usual Trip sequence. Again, all of this is in \cite{SMALL-TRIP}.

\subsection{Periodicity}
We would like to understand properties of periodicity of Trip sequences. In order to investigate periodicity of these sequences, it turns out to be much easier to work with eigenvectors of the matrices representing the $T$ that produces the Trip sequence.  This yields the following important proposition, which will prove indispensable to us in this paper.

\begin{proposition} Suppose $T$ is a composition of a finite number of Trip maps of the form  $T_{e,e,e}$,  $T_{e,(23),e}$, $T_{e,(132),(132)}$, $T_{(23),(23),e}$, or $T_{(13),(12),e}$. If $(1,\alpha,\beta)\in \triangle$ is an eigenvector of the matrix representation of  $T$, then $(\alpha,\beta)$ has a purely periodic triangle sequence of period $1$ with respect to $T$.
\end{proposition}

The proof is in sections six and seven of \cite{SMALL-TRIP}.   Many other Combo Trip maps have this property.  Unfortunately, as discussed in 
\cite{SMALL-TRIP}, not all do.


\section{A Family of Algorithms Yielding Periodicity}\label{mainsection}
We are now ready to present the main technical results of this paper. We begin by introducing the countable family of multidimensional continued fraction algorithms that we will use.

\subsection{A Countable Family of Algorithms}
The family of algorithms we consider is constructed from five Trip Maps, namely from maps 
$$\begin{array}{cccccc} T_{e,e,e} &  T_{e,(23),e} & T_{e,(132),(132)} & T_{(23),(23),e} & T_{(13),(12),e} \end{array}$$


From these we create the following three classes of maps, where $n \in \mathbb{Z}_{\ge0}$ :
\begin{enumerate}

\item Perform $T_{e,(132),(132)}$ $n$ times, then $T_{e,e,e}$ once
\item Perform $T_{e,(132),(132)}$ $n$ times, then $T_{e,(23),e}$ once
\item  Perform $T_{(13),(12),e}$ $n$ times, then $T_{(23),(23),e}$ once

\end{enumerate}

From now on, we refer to these as Class 1, Class 2, and Class 3. For each class, we get a different algorithm for each choice of non-negative integer $n$.  The domains can be recursively defined as follows. For Class 1, for a fixed $n$,  for each choice of $(n+1)$ tuple  $(m_1, \ldots m_n, k)$ of non-negative integers, we set  $\triangle_{m_{1}} $ as before to be the image of  $\triangle$ under $F_1^{m_1}(e,(132),(132)) F_0(e,(132),(132))$.  Then inductively define $\triangle_{(m_1, \ldots , m_l)},$  for $2\leq l\leq n,$ to be all $(x,y)\in \triangle$ such that 
$$T_{e,(132),(132)}^{(m_l)}(x,y) \in  \triangle_{(m_1, \ldots , m_{l-1})}$$ and finally set 
define $\triangle_{(m_1, \ldots , m_n, k)}$ to be all $(x,y)\in \triangle$ such that 
$$T_{e,e,e}^{(k)}(x,y) \in  \triangle_{(m_1, \ldots , m_{n})}.$$
The domains for Class 2 and Class 3 are defined in a similar fashion.

  To ease notation, we denote an $n$ tuple $(m_1, m_2, \ldots , m_n)$ by vector notation $\overline{m}$. Then to each point $(x,y)$, given a class and a nonnegative integer $n$,  we can associate an infinite sequence of $(n+1)$ tuples of nonnegative integers:$((\overline{m}_0,k_0), (\overline{m}_1,k_1), \ldots )$ defined as follows. First, we require
  $$(x,y)\in \triangle_{(\overline{m}_0,k_0)}.$$
  Now apply the given Combo Trip map (for example,  if we are doing Class 1,  perform $T_{e,(132),(132)}$ $n$ times, then $T_{e,e,e}$ once).  The initial point $(x,y)$ must be mapped to a point in $ \triangle_{(\overline{m}_1,k_1)}.$  Then applying the given Combo Trip map again, we must end up in  $\triangle_{(\overline{m}_2,k_2)}.$  Now continue.  As mentioned earlier, periodicity will imply that $x$ and $y$ are at worst cubic irrationals in the same cubic number field, as shown in \cite{SMALL-TRIP}.  (Of course, via simple concatenation, we could write the sequence of $(n+1)$ tuples as a sequence of nonnegative integers, if we so desire.)

Further, all of these maps can be explicitly defined on the appropriate domains as follows:
\[
T_{e,(132),(132)}^{(1)}(x,y)=\left(\frac{x}{1-x},\frac{y}{1-x}\right)
\]
\[
T_{e,e,e}^{(k)}(x,y)=\left(\frac{y}{x},\frac{1-x-ky}{x}\right)
\]

\[
T_{e,(23),e}^{(k)}(x,y)=\left(\frac{y}{x},\frac{(k+1)y+x-1}{x}\right)
\]
\[
T_{(13),(12),e}^{(k)}(x,y)=\left(\frac{x}{1-(k+1)y},\frac{y}{1-(k+1)y}\right)
\]
and $T_{(23),(23),e}^{(k)}(x,y)$ has first coordinate 
$$\frac{\left(1+\frac{1}{2}\left(-1+(-1)^{k}\right)\right)x-(-1)^{k}y}{x},$$
 and second coordinate
$$\frac{-1+\left(2+\frac{1}{2}\left(-1+(-1)^{k}\right)+\frac{1}{4}\left(1-(-1)^{k}+2k\right)\right)x+\left(-(-1)^{k}+\frac{1}{2}\left(-1+(-1)^{k}\right)\right)y}{x}.$$
Thus we can in principle compute the maps for each of the three classes of COMBO TRIP maps.

\subsection{Periodicity for Class 1 Maps}

\begin{theorem}\label{first}Let $A,B\in\mathbb{Z}$ with $A\geq0$ and $B\geq1.$ If $\alpha^{3}+A\alpha^{2}+B\alpha-1=0,$ then
$(\alpha,\alpha^{2})$ has a periodic triangle sequence under a map in Class 1.
\end{theorem}

\begin{proof}
The matrix representation of $T_{e,(132),(132)}^{(1)}$ is
\[
T_{e,(132),(132)}^{(1)}=\left(\begin{array}{ccc}
1 & 0 & 0\\
-1 & 1 & 0\\
0 & 0 & 1
\end{array}\right),
\]
Raising this to some integer power $B$ yields
 \[
(T_{e,(132),(132)})^{B}=\left(\begin{array}{ccc}
1 & 0 & 0\\
-B & 1 & 0\\
0 & 0 & 1
\end{array}\right).
\]
We know that 
\[
T_{e,e,e}^{(A)}=\left(\begin{array}{ccc}
0 & 0 & 1\\
1 & 0 & -1\\
0 & 1 & -A
\end{array}\right).
\]

The product $(T_{e,(132),(132)}^{(1)})^{B-1}\cdot T_{e,e,e}^{(A)}$ is 
\[
\left(\begin{array}{ccc}
0 & 0 & 1\\
1 & 0 & -B\\
0 & 1 & -A
\end{array}\right),
\]

which has eigenvector $(1,\alpha,\alpha^{2})$ where $\alpha^{3}+A\alpha^{2}+B\alpha-1=0,$
as desired.

\end{proof}

\subsection{Periodicity for Class 2 Maps}

\begin{theorem}\label{second}
Let $A,B\in\mathbb{Z}_{>0}.$ Then, if $\alpha^{3}-A\alpha^{2}-B\alpha+1=0,$
$(\alpha,\alpha^{2})$ has a periodic triangle sequence under a map in Class 2.\end{theorem}
\begin{proof}
We know that the matrix representation of $T_{e,(23),e}^{(A-1)}$ is 
\[
T_{e,(23),e}^{(A-1)}=\left(\begin{array}{ccc}
0 & 0 & -1\\
1 & 0 & 1\\
0 & 1 & A
\end{array}\right).
\]

The product $(T_{e,(132),(132)}^{(1)})^{B-1}\cdot T_{e,(23),e}^{(A-1)}$ is
\[
\left(\begin{array}{ccc}
0 & 0 & -1\\
1 & 0 & B\\
0 & 1 & A
\end{array}\right),
\]

which has eigenvector $(1,\alpha,\alpha^{2})$ where $\alpha^{3}-A\alpha^{2}-B\alpha+1=0,$
as desired. 

\end{proof}

\subsection{Periodicity for Class 3 Maps}

\begin{theorem}\label{third}
Let $A,B\in\mathbb{Z}_{>0},$ with $B>A$.  Then, if $\alpha^{3}-A\alpha+B\alpha-1=0,$
$(\alpha,\alpha-\alpha^{2})$ has a periodic triangle sequence under
a map in Class 3.\end{theorem}
\begin{proof}
The matrix representation of $T_{(13),(12),e}^{(0)}$ is 
\[
T_{(13),(12),e}^{(0)}=\left(\begin{array}{ccc}
1 & 0 & 0\\
0 & 1 & 0\\
-1 & 0 & 1
\end{array}\right),
\]

and
\[
(T_{(13),(12),e}^{(0)})^{A}=\left(\begin{array}{ccc}
1 & 0 & 0\\
0 & 1 & 0\\
-A & 0 & 1
\end{array}\right).
\]

We know that 
\[
T_{(23),(23),e}^{(2X-4)}=\left(\begin{array}{ccc}
0 & 0 & -1\\
1 & 1 & X\\
0 & -1 & -1
\end{array}\right).
\]

The product $(T_{(13),(12),e}^{(0)})^{A}\cdot T_{(23),(23),e}^{(2(1+B-A)-4)}$ is 
\[
\left(\begin{array}{ccc}
0 & 0 & -1\\
1 & 1 & 1+B-A\\
0 & -1 & -1+A
\end{array}\right),
\]

which has eigenvector $(1,\alpha,\alpha-\alpha^{2})$ where $\alpha^{3}-A\alpha^{2}+B\alpha-1=0,$
as desired.
\end{proof}

\subsection{Main Results on Units}

We want to see why the above cubic polynomials are actually quite natural.  Our first goal is 

\begin{theorem} 
Let $K$ be a cubic number field. If $u\in\mathcal{O}_{K}$ is a unit such that $0<u<1$,
then either $(u,u^{2}),$ $(u^{2},u^{4}),$ $(u,u^{2}-u),$ $(u^{2},u^{2}-u^{4}),$ $(uu',(uu')^{2}-uu'),$ or $((uu')^{2},(uu')^{2}-(uu')^{4}),$ where $u'$ is a conjugate of $u$, has a periodic triangle sequence under a map in Class 1, 2, or 3.
\end{theorem}

Now, by Dirichlet's Unit Theorem, every cubic number field contains an infinite number of units in the interval $(0,1)$. This means there are an infinite number of these ordered pairs.  This yields the following important corollary. 

\begin{corollary}
Let $K$ be a cubic number field, $u$ be a real unit  in $\mathcal{O}_K,$ with $0<u<1$ and $E=K(\sqrt{\Delta_{\mathbb{Q}(u)}})$ (where  $\Delta_{\mathbb{Q}(u)}$ is the discriminant of $\mathbb{Q}(u)$).  Then there exists a point $(\alpha,\beta)$, with $\alpha,\beta\in E$ irrational, such that $(\alpha,\beta)$ has a periodic TRIP sequence. 
\end{corollary}

This theorem and corollary form the essence of this paper.  The proof of the theorem reduces to calculations, which, though not difficult, requires a number of rather technical lemmas, which are stated below. 

\begin{lemma}\label{root}Let $K$ be a cubic number field. Let $u$ be a unit in $\mathcal{O}_{K}.$  Then there exists $A,B\in \mathbb{Z}$ such that either $u^{3}+Au^{2}+Bu+1=0$ or $u^{3}+Au^{2}+Bu-1=0$.
\end{lemma}

\begin{proof}Since we know $u$ is an algebraic integer, there must exists integers $A,B,C$ such that $u^{3}+Au^{2}+Bu+C=0$.  This means $Cu^{-3}+Bu^{-2}+Au^{-1}+1=0$.  Since $u^{-1}$ must also be an algebraic integer, we know $C=\pm1.$ This, implies $u^{3}+Au^{2}+Bu+1=0$ or $u^{3}+Au^{2}+Bu-1=0$. 

\end{proof}

We now use Lemma \ref{root} in order to prove Lemma \ref{equations}.

\begin{lemma}\label{equations} Let $\alpha\in\mathbb{R}$ be a cubic irrational. Let $u$ be a unit in $\mathcal{O}_{\mathbb{Q}(\alpha)}$ and $\Delta$ be the discriminant of the cubic number field $\mathbb{Q}(\alpha).$  There exist $P,Q \in \mathbb{Z}_{>0}$ such that some element, strictly between zero and one and in $\mathbb{Q}(\alpha,\sqrt{\Delta})$,  satisfies one of the following equations:
\begin{enumerate}
\item \label{firstequation}$x^{3}+Px^{2}+Qx-1=0$
\item \label{secondequation}$x^{3}-Px^{2}-Qx+1=0$
\item \label{thirdequation} $x^{3}+Qx-1=0$
\item \label{fourthequation}$x^{3}-Px^{2}+Qx-1=0$ where $Q>P.$
\end{enumerate}
\end{lemma}
\begin{proof}
Let $0<u<1$ be an irrational unit in $\mathbb{Q}(\alpha).$ From Lemma \ref{root} we know that there must exist $A,B \in \mathbb{Z}_{\ge0}$ such that one of $f_{\delta_1,\delta_2,\delta_3}(u)=u^{3}\pm Au^{2}\pm Bu \pm1=0$ must hold,  where $\delta_1,\delta_2,$ and $\delta_3$ specify the signs in front of $A,B,$ and $1$. Let $f_{\epsilon_1,\epsilon_2,\epsilon_3}$ be a specific one of these polynomials such that $f_{\epsilon_1,\epsilon_2,\epsilon_3}(u)=0$.  
This yields $18$ cases. 

We first consider the cases when $A,B\ne 0$.
\begin{case} $\epsilon_1=+,\epsilon_2=+,\epsilon_3=+$
\end{case}
In this case, $f_{\epsilon_1,\epsilon_2,\epsilon_3}(u)=u^{3}+ Au^{2}+ Bu +1$. Since this polynomial cannot have any roots between $0$ and $1$, this yields a contradiction, and so this case cannot occur.

\begin{case}$\epsilon_1=+,\epsilon_2=+,\epsilon_3=-$
\end{case}
This means $f_{\epsilon_1,\epsilon_2,\epsilon_3}$ is in the form of (\ref{firstequation}) and so we are done.

\begin{case} $\epsilon_1=+,\epsilon_2=-,\epsilon_3=+$
\end{case}
This means $f_{\epsilon_1,\epsilon_2,\epsilon_3}(u)=u^{3}+Au^{2}- Bu +1=0$, which implies $u(u^{2}+Au-B)=-1$.  This means
\begin{eqnarray*}
u^{2}&=&\frac{1}{u^{4}+A^{2}u^{2}+B^{2}+2Au^{3}-2Bu^{2}-2ABu}\\
  &=&\frac{1}{u^{4}+(A^{2}-2B)u^{2}+B^{2}+2A(u^{3}-Bu)}\\
 &=&\frac{1}{u^{4}+(A^{2}-2B)u^{2}+B^{2}+2A(-1-Au^{2})}\\
 &=&\frac{1}{u^{4}+(-A^{2}-2B)u^{2}+B^{2}-2A}
\end{eqnarray*}
Setting $v=u^2$, we have 
\[
v=\frac{1}{v^{2}+(-A^{2}-2B)v+B^{2}-2A}
\]
This implies $v^{3}+(-A^{2}-2B)v^{2}+(B^{2}-2A)v-1=0.$ Since we know $0<u<1$, this means $0<v<1$ too.  So for  $v^{3}+(-A^{2}-2B)v^{2}+(B^{2}-2A)v-1=0$ to have a solution between $0$ and $1$ we need $(B^{2}-2A)>0$.  Letting $P=A^2+2B$ and $Q=B^{2}-2A$, we are left with $v^3-Pv^2+Qv-1=0$ where $P,Q\in \mathbb{Z}^+$. This reduces to Case 6.

\begin{case}$\epsilon_1=+,\epsilon_2=-,\epsilon_3=-$\end{case}
This means $u^{3}+Au^{2}-Bu-1=0$, so $u(u^{2}+Au-B)=1$. As before, we have
\begin{eqnarray*}
u^{2}&=&\frac{1}{u^{4}+A^{2}u^{2}+B^{2}+2Au^{3}-2Bu^{2}-2ABu}\\
  &=&\frac{1}{u^{4}+(A^{2}-2B)u^{2}+B^{2}+2A(u^{3}-Bu)}\\
 &=&\frac{1}{u^{4}+(A^{2}-2B)u^{2}+B^{2}+2A(1-Au^{2})}\\
 &=&\frac{1}{u^{4}+(-A^{2}-2B)u^{2}+B^{2}+2A}
\end{eqnarray*}
which then implies, setting $v=u^2$, that   $v^{3}+(-A^{2}-2B)v^{2}+(B^{2}+2A)v-1=0.$ This reduces to Case 6.

\begin{case} $\epsilon_1=-,\epsilon_2=+,\epsilon_3=+$
\end{case} 
This means $u^{3}-Au^{2}+Bu+1=0$ and so $ u(u^{2}-Au+B)=-1$.  As before, we have
\begin{eqnarray*}
u^{2}&=&\frac{1}{u^{4}+A^{2}u^{2}+B^{2}-2Au^{3}+2Bu^{2}-2ABu}\\
  &=&\frac{1}{u^{4}+(A^{2}+2B)u^{2}+B^{2}-2A(u^{3}+Bu)}\\
 &=&\frac{1}{u^{4}+(A^{2}+2B)u^{2}+B^{2}-2A(-1+Au^{2})}\\
 &=&\frac{1}{u^{4}+(2B-A^{2})u^{2}+B^{2}+2A}
\end{eqnarray*}
which then implies, setting $v=u^2$, that $v^{3}+(2B-A^{2})v^{2}+(B^{2}+2A)v-1=0.$   If $2B-A^{2}\geq0,$ then $g_{\gamma_1,\gamma_2,\gamma_3}$ is of desired forms (\ref{firstequation}).   If $2B-A^{2}<0$, then, as in the previous case, $g_{\gamma_1,\gamma_2,\gamma_3}$ must be of desired form (\ref{fourthequation}).

\begin{case}$\epsilon_1=-,\epsilon_2=+,\epsilon_3=-$
\end{case}
It is here where we will see our need to consider units not in just the original cubic number field but possibly in a quadratic extension.  We have that $u$ is a real unit between zero and one that is a root of the irreducible polynomial
$$f(x) = x^3 -Ax^2+Bx-1.$$
We consider the three subcases of $A=B$, $A<B$ and $A>B$.  If $A=B$, then $f(x)$ is not irreducible, since
we have 
$$x^{3}-Ax^{2}+Ax-1=(x-1)(x^2-(A-1)x+1).$$
Hence this cannot happen.

If $B>A$,  $f_{\epsilon_1,\epsilon_2,\epsilon_3}$ is of desired form (\ref{fourthequation}) and we are done.

This leaves the case for when $A>B$.   Label the three roots of $f(x)$ by our original $u$, and $u_2$ and $u_3$.  We first will show that one of these two other roots is also real, between zero and one.  Note that 
$f(0)=-1<0$ and
$$f(1) = 1 - A + B -1 <0.$$
(This is the step where we are using that $A>B$.)  
Since $f(u)=0$ and since $u$ is not a double root (using here that $f(x)$ is irreducible), we can indeed assume that $u_2$ is real between zero and one.   We know that 
$$uu_2u_3=1,$$
meaning that $u_3$ is real and greater than one. 
We know that $1/u_3 = uu_2$ is real between zero and one  and will have minimal polynomial 
$$x^{3}-Bx^{2}+Ax-1,$$
reducing our problem to the subcase where $B>A.$  Also, note that even with $u$ being a unit in a cubic number field $K$, we can only claim that $u_3$ and hence $1/u_3$ are units in the quadratic extension $K(\sqrt{\Delta_{\mathbb{Q}(u)}})$

\begin{case}$\epsilon_1=-,\epsilon_2=-,\epsilon_3=+$
\end{case}
If $u^{3}-Au^{2}-Bu+1=0,$ then $f_{\epsilon_1,\epsilon_2,\epsilon_3}$ is of desired form (\ref{secondequation}). 

\begin{case} $\epsilon_1=-,\epsilon_2=-,\epsilon_3=-$
\end{case}
The equation $u^{3}-Au^{2}-Bu-1=0$ will not have any solutions between $0$ and $1$ so this case cannot occur.

We now consider the cases when $B=0$. 
\begin{case}$\epsilon_1=+,B=0,\epsilon_3=+$ 
\end{case}
The equation $u^{3}+Au^{2}+1=0$ will not have any solutions between $0$ and $1$ so this case cannot occur.

\begin{case}$\epsilon_1=+,B=0,\epsilon_3=-$ 
\end{case}
This means $u^{3}+Au^{2}-1=0$, which implies $u(u^{2}+Au)=1$. This means
\begin{eqnarray*}
u^{2}&=&\frac{1}{u^{4}+A^{2}u^{2}+2Au^{3}}\\
  &=&\frac{1}{u^{4}+A^{2}u^{2}+2A(1-Au^{2})}\\
 &=&\frac{1}{u^{4}-A^{2}u^{2}+2A}.
 \end{eqnarray*}
 As before, set $v=u^2$, which means that 
 $v^{3}-A^{2}v^{2}+2Av-1=0.$ We are now in case 6.

\begin{case}$\epsilon_1=-,B=0,\epsilon_3=+$ 
\end{case}
If $u^{3}-Au^{2}+1=0$, then $u(u^{2}-Au)=-1$ and so 
\begin{eqnarray*}
u^{2}&=&\frac{1}{u^{4}+A^{2}u^{2}-2Au^{3}}\\
  &=&\frac{1}{u^{4}+A^{2}u^{2}-2A(-1+Au^{2})}\\
 &=&\frac{1}{u^{4}-A^{2}u^{2}+2A}\\
 \end{eqnarray*}
As in the previous cases, letting $v=u^2$, yields the equation $v^{3}-A^{2}v^{2}+2Av-1=0,$ which reduces to Case 6.

\begin{case}$\epsilon_1=-,B=0,\epsilon_3=-$ 
\end{case}
The equation $u^{3}-Au^{2}-1=0$ has no solutions between $0$ and $1$, so this case cannot occur.

Next, we consider the cases when $A=0$.

\begin{case}$A=0,\epsilon_2=+,\epsilon_3=+$\end{case} 
The equation  $u^{3}+Bu+1=0$ has no solutions between $0$ and $1$, so this case cannot occur.

\begin{case}$A=0,\epsilon_2=+,\epsilon_3=-$\end{case} 
If $u^{3}+Bu-1=0,$ then $f_{\epsilon_1,\epsilon_2,\epsilon_3}$ is of desired form (\ref{thirdequation}).

\begin{case}$A=0,\epsilon_2=-,\epsilon_3=+$\end{case} 
Suppose that $u^{3}-Bu+1=0.$ Then $u(u^{2}-B)=-1$ and so 
$$u^{2}=\frac{1}{u^{4}-2Bu^{2}+B^{2}}$$
Then, letting $v=u^2$, we have  $v^{3}-2Bv^{2}+B^{2}v-1=0$, which is case 6.
\begin{case}$A=0,\epsilon_2=-,\epsilon_3=-$\end{case}
The equation  $u^{3}-Bu-1=0$ has no solutions between $0$ and $1$, so this case cannot occur.

\begin{case} $A=0,B=0, \epsilon_3=+$\end{case}
The equation  $u^{3}-1=0$ has no solutions between $0$ and $1$, so this case cannot occur.

\begin{case} $A=0,B=0, \epsilon_3=-$\end{case}
The equation  $u^{3}+1=0$ has no solutions between $0$ and $1$, so this case cannot occur.

\end{proof}

Thus, we have now showed that for every cubic number field $K$, some element $u$ in a quadratic extension of $K$ will be the solution to one of $4$ classes of equations.  We next want to examine the periodicity properties of solutions of these four equations.  Our goal is to be able to use $u$ to produce points that have periodic triangle sequences under a map in one of our three classes.  
We are now ready for the proof of the main theorem of this section, which will now be quite simple. 

\begin{proof}
Let $u\in\mathcal{O}_{K}$ be a unit such that $0<u<1$.  Then by Lemma \ref{equations} either $v=u,$ $v=u^2,$ or $v=uu'$ for some conjugate $u'$ of $u$ is a solution to equations (\ref{firstequation}), (\ref{secondequation}), (\ref{thirdequation}), or (\ref{fourthequation}).  If $v$ is a solution to (\ref{firstequation}), (\ref{secondequation}), (\ref{thirdequation}), then by Theorem \ref{first} or Theorem \ref{second}, the triangle sequence of the point $(v,v^2)$ is periodic with respect to a map in Class 1 or Class 2. If $v$ is a solution to (\ref{fourthequation}), then by Theorem \ref{third}, the triangle sequence of $(v,v-v^2)$ is periodic with respect to a map in Class 3. 

\end{proof}

Of course, the earlier work of Dubois and Paysant-Le Roux \cite{Dubois-Paysant-Le Roux75} shows that every real cubic number field has a pair of elements that is purely periodic with respect to Jacobi-Perron (which is the same thing as periodic with respect to the G\"{u}ting map, which is in turn a type of Combo Trip map.).  The above result shows that the same happens for our three classes, and as discussed earlier, the involved polynomials and associated pairs of real numbers are simpler to write down and include far more polynomials and pairs of real numbers.

\section{ Hermite Matrices: Periodicity of a Row and the  Hermite Problem}\label{hermitesection}
We will now see how the above results present an approach to the Hermite problem.  An ideal solution to the Hermite problem would be that we could find a single algorithm $T$ such that given an arbitrary $\alpha$ we could produce a $\beta$ such that the sequence of integers associated with $(\alpha,\beta)$ with respect to $T$ is periodic if and only if $\alpha$ is a cubic irrational.  However, as discussed in the introduction,  it seems as though no such single algorithm can exist.

As mentioned in the introduction, we think that  the best possible solution to the Hermite problem is in the form of a family of algorithms,  meaning that for any arbitrary pair  $(\alpha, \beta)$ of real numbers, we can produce a sequence of integers associated to   $(\alpha,\beta)$ that will be periodic with respect to some algorithm in the family if and only if $(\alpha, \beta) $ are in the same cubic number field. Such a family is capable of being encoded as a matrix $(a_{ij})$, with $1\leq i,j <\infty$.  Each row will be the sequence of integers associated to some algorithm in the family.   In this language, we would want one of the rows to be eventually periodic if and only if $\alpha$ is a cubic irrational.  Such is matrix will now be constructed. 

Suppose $\mathcal{F}$ is a countable family of multidimensional continued fraction algorithms.  List them as $T(1),T(2), \ldots.$ Given a pair of reals $(\alpha,\beta)$, for each algorithm $T(n)$ there will be a corresponding sequence of integers, which we will denote by $(a_{n1}, a_{n2},a_{n3}, \dots).$

\begin{definition} The Hermite matrix $\mathcal{H}(\alpha, \beta )$ for a countable family of multidimensional continued fraction algorithms is the matrix with non-negative integer entries whose $n$th row is the $(a_{n1}, a_{n2},a_{n3}, \dots)$ with respect to the $T(n)$-algorithm.
\end{definition}
Thus 
$$\mathcal{H}(\alpha, \beta) =\left( \begin{array}{c}
\mbox{the $T(1)$ sequence for $(a(1),b(1))$}   \\
\mbox{the $T(2)$ sequence for $(a(2),b(2))$} \\
\mbox{the $T(3)$ sequence for $(a(3),b(3))$} \\
\vdots \end{array}  \right)$$

Theorems \ref{first},\ref{second} and \ref{third} can be interpreted as giving criterion for when the corresponding Hermite matrices have a periodic row.

This also lead to 
\begin{question} Let $\mathcal{T}$ be the family of multidimensional continued fractions formed from COMBO Trip maps.  Is it the case that a pair of reals $(\alpha, \beta)\in \triangle$ are in the same cubic number field if and only if there is a row in the corresponding Hermite matrix that is eventually periodic?\end{question}

The rhetoric of Hermite matrices allows us to reformulate conjecture 36 in \cite{Karpenkov13}:
\begin{conjecture}  There is a family of multidimensional continued fractions spanned by a finite number of multidimensional continued fraction algorithms such that a pair of reals $(\alpha, \beta)\in \triangle$ are in the same cubic number field if and only if there is a row in the corresponding Hermite matrix is eventually periodic.\end{conjecture}

Of course, there are corresponding Hermite matrices for higher dimensional analogs.

\section{Questions}
A lot of questions remain.  First, what other families of multi-dimensional continued fractions exist whose associated matrix   has a periodic row if and only if $\alpha$ is cubic?  Is there there a way to stay in the initial cubic number field $\Q(\alpha)$ and not go the larger number field $\Q(\alpha, \sqrt{\triangle})$?  Can we require all the elements in a family of multi-dimensional continued fractions to fall into the scope of such maps as in \cite{Lagarias93} .

Once having chosen our family of Trip maps, there are then many questions about the corresponding Hermite  matrix $\mathcal{H}(\alpha, \beta)$ which are in direct analog to what is known about continued fraction expansions. These questions strike us as currently accessible.  

The matrix $\mathcal{H}(\alpha, \beta)$  should contain  all information about the numbers $\alpha$ and $\beta$.  What other algebraic properties of $\alpha$ and $\beta$  can be obtained from the Hermite matrix?

Of course, the original Hermite problem is still open.  In light of this paper, we think the correct line of attack would be in showing that no single multidimensional continued fraction algorithm will have periodicity being equivalent to a number being a cubic.  This strikes us as quite hard.

\end{document}